\documentclass[10pt]{amsart}

\usepackage{amssymb}
\usepackage{amsthm}
\usepackage{mathtools}
\usepackage{stmaryrd}
\usepackage{mathrsfs}
\usepackage{tikz}
\usepackage{tikz-cd}
\usetikzlibrary{arrows}

\newtheorem{thm}{Theorem}[section]
\newtheorem{lem}[thm]{Lemma}
\newtheorem{prop}[thm]{Proposition}
\newtheorem{cor}[thm]{Corollary}

\theoremstyle{definition}

\newtheorem{defn}[thm]{Definition}

\newtheorem{examps}[thm]{Examples}

\newcommand{\si}{^{(\infty)}}

\theoremstyle{remark}
\newtheorem{rem}[thm]{Remark}

\numberwithin{equation}{section}

\newcommand{\bC}{{\mathbb C}}

\newcommand{\bF}{{\mathbb F}}

\newcommand{\bN}{{\mathbb N}}

\newcommand{\bZ}{{\mathbb Z}}

\newcommand{\cC}{{\mathcal C}}

\newcommand{\cH}{{\mathcal H}}

\newcommand{\cO}{{\mathcal O}}

\newcommand{\cT}{{\mathcal T}}
\newcommand{\cU}{{\mathcal U}}

\newcommand{\cZ}{{\mathcal Z}}

\newcommand{\acts}{\curvearrowright}

\DeclareMathOperator{\id}{id}

\DeclareMathOperator{\Aut}{Aut}

\DeclareMathOperator{\Ker}{ker}

\begin{document}

\title{Selfless Reduced Crossed Product $C^{*}$-algebras Arising from Almost Periodic Actions}

\author[S. Ohshima]{Syuichi Ohshima}
\thanks{Department of Mathematics, Faculty of Science, Hokkaido University,
Kita 10, Nishi 8, Kita-Ku, Sapporo, Hokkaido, 060-0810, Japan.
E-mail: \texttt{oshima.shuichi.c4@elms.hokudai.ac.jp}}

\begin{abstract}
We prove that the reduced crossed product arising from an almost
periodic action of an infinite countable discrete group on a unital, separable
$C^*$-algebra is selfless whenever the crossed product is simple
and monotracial and the group admits a topologically free extreme
boundary action.
We also give an alternative proof of a special case of a result
of Gao, Junge, Kunnawalkam Elayavalli, Patchell, and Robert:
the direct product of an exact $C^*$-simple group and a
$C^*$-selfless group is $C^*$-selfless.
\end{abstract}

\maketitle%

\markboth{S. OHSHIMA}{Selfless Reduced Crossed Product $C^{*}$-Algebras Arising from Almost Periodic Actions}


\section{Introduction}

Strict comparison, introduced by Blackadar in \cite{Bla88}, is a fundamental regularity property of $C^*$-algebras and plays a central role in Elliott's classification program (see, e.g., \cite{CGSTW23, CETW22, EGLN25, GLN20a, GLN20b, Tom08, Win18}). Despite its importance, establishing strict comparison is often highly nontrivial. Indeed, although the notion dates back to 1988, it was not until 2024 that strict comparison was established for the reduced group $C^*$-algebras of the free groups $\mathbb F_n$ with $2\leq n<\infty$. (The case of $\mathbb F_\infty$ had previously been settled by Robert \cite[Proposition~6.3.2]{Rob12}, who attributed the result to R{\o}rdam.) Amrutam, Gao, Kunnawalkam Elayavalli, and Patchell \cite{AGKP25} resolved this question by proving a stronger property: every finitely generated acylindrically hyperbolic group with trivial finite radical and the rapid decay property is \emph{$C^*$-selfless}, meaning that its reduced group $C^*$-algebra is selfless with respect to the canonical trace.

The main focus of this paper is the notion of a \emph{selfless $C^{*}$-probability space}, introduced by Robert in \cite{Rob25}. 
Notably, for a tracial $C^{*}$-probability space, selflessness implies simplicity, uniqueness of the trace, stable rank one, and strict comparison. 
We say that a $C^{*}$-probability space $(A,\phi)$ is \emph{selfless} if $A\not\simeq \bC$ and the first factor embedding $\theta$ into the reduced free product $C^{*}$-probability space is \emph{existential}; that is, there exist an ultrafilter $\cU$ (on a set $I$) and an embedding $\sigma : (A, \phi)*(A, \phi)\to (A^{\cU}, \phi^{\cU})$ such that $\sigma\circ\theta$ agrees with the diagonal embedding $\iota$ of $(A, \phi)$ into $(A^{\cU}, \phi^{\cU})$:

\[\begin{tikzcd}
	{(A,\phi)} && {(A^{\cU},\phi^{\cU})} \\
	& {(A,\phi)*(A,\phi)}
	\arrow["\iota", hook, from=1-1, to=1-3]
	\arrow["{\theta}", hook, from=1-1, to=2-2]
	\arrow["\sigma", hook, from=2-2, to=1-3]
\end{tikzcd}\]

By \cite[Theorem~2.6]{Rob25}, $(A,\phi)$ is selfless if and only if
the first-factor embedding
\[
    (A,\phi)\hookrightarrow(A,\phi)*(C,\kappa)
\]
is existential for some $C^*$-probability space $(C,\kappa)$
with $C\neq\bC$.
Selflessness has been established for a wide range of
$C^*$-probability spaces.
For instance, every simple, exact, $\cZ$-stable, monotracial
$C^*$-probability space is selfless
\cite[Theorem~3]{Ozw25}.
We refer the reader to
\cite{AGKP25, Ari26, BS26, FKCP25, FKCP26, GJEPR26, GKPT26, HKPR25, HKR25,
KPT25, RTV25, Rob25, Vig26, Yan25}
for further examples.

Selflessness of reduced group $C^*$-algebras has received
particular attention; see, for example,
\cite{AGKP25, Ari26, Ozw25, Vig26}.
Ozawa \cite[Theorem~1]{Ozw25} proved that every infinite countable discrete
group admitting a topologically free extreme boundary action
is $C^*$-selfless.
This includes non-elementary free product groups.
Since selflessness implies simplicity, every $C^*$-selfless group
is $C^*$-simple.
Although many known $C^*$-simple groups have been shown to be
$C^*$-selfless, the converse remains open in general.
For linear groups, however, Vigdorovich \cite{Vig26}
proved that the two properties are equivalent.
Regarding permanence properties, Gao, Junge,
Kunnawalkam Elayavalli, Patchell, and Robert
\cite[Corollary~7.9]{GJEPR26} proved that the direct product
of an exact $C^*$-simple group and a $C^*$-selfless group
is $C^*$-selfless.
We provide an alternative proof of this result in
Corollary~\ref{cor:direct-product}.

More generally, one may ask under what conditions a simple, monotracial reduced crossed product is selfless. A result in this direction is the following theorem of Gao, Junge, Kunnawalkam Elayavalli, Patchell, and Robert.
\begin{thm}\textup{\cite{GJEPR26}}
If an infinite countable discrete group $\Gamma$ with the property $PHP$ acts on a unital, simple, monotracial $C^{*}$-probability space $\left(A, \phi\right)$ by approximately inner automorphisms, then the reduced crossed product $\left(A\rtimes_{r}\Gamma, \phi\circ E_{A}\right)$ is selfless. 
\end{thm}

Their approach is based on the framework of selfless completely positive maps and an adaptation of Ozawa's PHP technique. Approximate innerness plays a crucial role in their argument.
If the action is approximately inner, every ideal of the coefficient algebra is invariant. Thus, such actions cannot yield selfless crossed products unless the coefficient algebra is simple. In particular, the above theorem cannot be used to obtain selfless crossed products with commutative coefficient algebras other than $\bC$.

To move beyond this restriction, we consider almost periodic actions without assuming approximate innerness. In this setting, the coefficient algebra need not be simple; instead, we assume that it is $\Gamma$-simple and admits a unique $\Gamma$-invariant tracial state.
We note that, when $\Gamma$ is $C^*$-simple, $A$ is $\Gamma$-simple and admits a unique $\Gamma$-invariant tracial state if and only if $A\rtimes_r\Gamma$ is simple and monotracial; see \cite[Theorem~1.8 and Remark~4.2]{BKKO17}.
Our main result is the following.
\begin{thm}
Let $\Gamma$ be an infinite countable discrete group with a topologically free
extreme boundary, and let $\alpha\colon\Gamma\curvearrowright A$
be an almost periodic action on a unital, separable $C^*$-algebra.
Suppose that $A$ is $\Gamma$-simple and admits a unique
$\Gamma$-invariant tracial state $\phi$.
Then $(A\rtimes_{\alpha,r}\Gamma,\phi\circ E_A)$ is selfless.
\end{thm}
For a group action on a $C^{*}$-algebra $\sigma:\Gamma\curvearrowright A$, 
if $\overline{\sigma\left(\Gamma\right)}<\Aut\left(A\right)$ is compact in the point-norm topology,
then $\sigma$ is called an almost periodic action.

For example, a dense embedding $\mathbb F_2\hookrightarrow SO(3)$ induces a minimal almost periodic action of $\mathbb F_2$ on $S^2$, with the normalized area measure as its unique invariant probability measure. Our theorem therefore implies that $C(S^2)\rtimes_r\mathbb F_2$ is selfless. Further examples are given in Section~4.

Related results on selflessness of crossed products have been obtained in several settings. For example, Bell \cite{Bel26} established complete selflessness for strictly ergodic actions of groups with property PHP on the Cantor set satisfying dynamical comparison. Flores, Klisse, \'{O} Cobhthaigh, and Pagliero \cite{FKCP26} introduced an action-dependent version of the PHP property and used it to obtain selfless inclusions of reduced group $C^*$-algebras into reduced crossed products, also allowing twists. Their examples include boundary and projective actions, as well as certain actions yielding tracial crossed products. Recently, Toyota \cite{Toy26} proved complete selflessness for crossed products arising from ergodic almost periodic actions of groups with property PHP on unital, separable $C^*$-algebras equipped with a faithful invariant state, assuming that the image of the action has connected closure in the point-norm topology. His results include the above action on $S^2$.

Our approach does not require the closure of the image of the action to be connected, but assumes that the acting group admits a topologically free extreme boundary. In Section~4, we illustrate the scope of our theorem with examples involving both commutative and noncommutative coefficient algebras.

\vspace{\baselineskip}

\noindent\textbf{Outline of the proof.}

First, we adapt Ozawa's ``tree-graded'' space method to show that
the canonical inclusion
\[
    A\rtimes_{\sigma,r}\Gamma
    \subset
    A\rtimes_{\sigma*\mathrm{id},r}(\Gamma*\mathbb F_2)
\]
is existential.

Since $A\rtimes_{\sigma,r}\Gamma$ is unital, simple, and monotracial,
we obtain a state-preserving embedding
\[
    \bigl(A\rtimes_{\sigma,r}\Gamma,\phi\circ E_A\bigr)
    *(\mathcal C_1,\omega)
    \hookrightarrow
    \bigl(
        A\rtimes_{\sigma*\mathrm{id},r}(\Gamma*\mathbb F_2),
        \phi\circ E_A
    \bigr)^{\mathcal U}
\]
whose restriction to $A\rtimes_{\sigma,r}\Gamma$ is the diagonal
embedding.
Combining this embedding with the existentiality of the first inclusion
and passing to an iterated ultrapower, we conclude that the canonical
inclusion
\[
    \bigl(A\rtimes_{\sigma,r}\Gamma,\phi\circ E_A\bigr)
    \subset
    \bigl(A\rtimes_{\sigma,r}\Gamma,\phi\circ E_A\bigr)
    *(\mathcal C_1,\omega)
\]
is existential.

\vspace{\baselineskip}

\noindent\textbf{Notation.}

\begin{itemize}
\item A \emph{$C^{*}$-probability space} is a pair $(A,\phi)$ 
consisting of a unital $C^{*}$-algebra $A$ 
and a distinguished state $\phi$ that is GNS-faithful, 
unless the $C^{*}$-algebra 
under consideration is an ultrapower or an ultraproduct.
\item Amalgamated free products of $C^{*}$-algebras 
with nondegenerate conditional expectations are 
always taken in the reduced sense.
\item Let $\cU$ be an ultrafilter on a set $I$. 
We denote by $A^{\cU}$ the ultrapower of a $C^{*}$-algebra $A$, 
and by $[a_{i}]_\cU$ the element of $A^{\cU}$ represented 
by a bounded family $(a_{i})_{i\in I}$ in $A$.
\item For the reduced crossed product $C^*$-algebra $A\rtimes_{r}\Gamma$, let $E_{A}:A\rtimes_{r}\Gamma\to A$ denote the canonical conditional expectation, 
which is given by $E_{A}(a\lambda_s)=a\delta_{s,e}$ for $a\in A$ and $s\in\Gamma$.
\end{itemize}

\vspace{\baselineskip}

\noindent\textbf{Acknowledgment}:
This is a part of the author's master's thesis, written under the supervision of Professor Yuhei Suzuki at Hokkaido University. The author is deeply grateful to Professor Suzuki for his helpful comments and continuous support. 
 
 \vspace{\baselineskip}
 
 \noindent\textbf{AI statement}:
Gemini 3.1 Pro and ChatGPT 6.0 Astra were used for English editing. ChatGPT 6.0 Astra also suggested the construction of a net of finite tuples of unitaries used in Lemma~\ref{thm:strong}, which enabled the removal of the $\cZ$-stability assumption and inspired the author's subsequent generalization in Proposition~\ref{thm:verystrong}. It was also used to help identify potential gaps in the proofs. The author has independently verified all mathematical arguments and takes sole responsibility for the paper.

\section{Preliminaries}

\subsection{Extreme boundary action}

In this subsection, we recall the notion of an extreme boundary action of a discrete group.

\begin{defn}

Let $\Gamma$ be an infinite countable discrete group, 
and let $\Gamma\curvearrowright X$ 
be an action on a compact topological space $X$. 
The action $\Gamma\curvearrowright X$ is called
 an \emph{extreme boundary action} if it is minimal 
 and \emph{extremely proximal}, in the sense 
 that for any pair of non-empty open subsets $U$ 
 and $V$ of $X$, 
 there exists $g\in \Gamma$ 
 such that $g(X\setminus U)\subset V$.

\end{defn}

For a compact $\Gamma$-space $X$
 and $x\in X$, 
 let $\Gamma_{x}^{\circ}$ denote the subgroup consisting of elements of 
 $\Gamma$ that act as the identity on some neighborhood of $x$. 
 The compact $\Gamma$-space $X$ is said to be 
 \emph{topologically free} if $\Gamma_{x}^{\circ}=\{1\}$ for all $x\in X$.
Note that if an infinite countable discrete group $\Gamma$ admits a topologically free extreme boundary, then it admits a second countable one. 

We say that a sequence $(z_n)_{n}$ in $\Gamma$ is \emph{axial} 
if there exists a topologically free extreme boundary action $\Gamma\acts X$
 and distinct points $z_+, z_- \in X$ satisfying the following two conditions:
\begin{enumerate}
\item For any neighborhoods $U_+$ of $z_+$ and $U_-$ of $z_-$, 
one has $z_n (X\setminus U_-) \subset U_+$ 
for all sufficiently large $n$ 
(or, equivalently, $z_n^{-1}(X \setminus U_+) \subset U_-$ for all sufficiently large $n$).
\item The $\Gamma$-action on $\{z_+, z_-\}$ is free, in the sense that 
$g\{z_+, z_-\}\cap\{z_+, z_-\}\neq\emptyset$ implies $g=1$.
\end{enumerate}
Note that a second countable extreme boundary $X$ with $|X|>2$ admits an axial sequence if and only if the action on $X$ is topologically free; see \cite{Ozw25}. If $\Gamma$ is infinite, topological freeness implies that $X$ is infinite, so the condition $|X|>2$ is automatic.

We remark that a discrete group admitting a topologically free extreme boundary is 
$C^{*}$-simple by \cite[Theorem 1.5]{KK17} and \cite[Theorem 2.3]{Gla74}.

Many geometric groups admit an extreme boundary.

\begin{thm}\textup{{\cite[Theorem 1.3]{Yan25}}}\label{thm:boundary}
Let $\Gamma$ be a non-elementary acylindrically hyperbolic group.  
Then $\Gamma$ admits an extreme boundary action on a compact metrizable space. 
If $\Gamma$ contains no nontrivial finite normal subgroup, 
then the action is topologically free.     
\end{thm}

\section{Main Result}
In this section, we prove the main result. We begin with a few observations.
\begin{lem}\label{thm:group}
Let $(A,\phi)$ be a separable $C^{*}$-probability space, 
and let $\Gamma$ be an infinite countable discrete group 
which has a topologically free, second countable extreme boundary $X$. If
$\sigma:\Gamma\curvearrowright A$ is an almost periodic action, then
there exists an axial sequence $(z_n)_{n\in\mathbb N}\subset\Gamma$
such that
\[
\sigma_{z_n}\longrightarrow \mathrm{id}.
\]
\end{lem}

\begin{proof}
Let $(w_{n})_{n}\subset\Gamma$ be an axial sequence with 
distinct points $z_{+}, z_{-}\in X$. 
Since $\sigma: \Gamma\curvearrowright A$ is almost periodic, 
by passing to a subsequence if necessary, 
we may assume that $\sigma_{w_{n}} \to \alpha$ for some $\alpha\in \operatorname{Aut}(A)$. 
Let $(V_{n,+})_{n}$ and $(V_{n,-})_{n}$ be decreasing neighborhood bases of 
$z_+$ and $z_-$, respectively, such that $V_{n, +}\cap V_{n,-}=\emptyset$ for every $n$.

Suppose that there exists an axial sequence $(x_{n})_{n} \subset \Gamma$ 
with distinct points $x_{+}, x_{-} \in X$ 
such that $x_{n}^{2}=1$ for each $n$.
Since the action $\Gamma \curvearrowright X$
 is topologically free, there exists a point $x_{0} \in X \setminus \{x_{+}, x_{-}\}$. 
Let $W_{+}$ and $W_{-}$ 
be disjoint open neighborhoods of 
$x_{+}$ and $x_{-}$ in $X \setminus \{x_{0}\}$, respectively. 
Then, for all sufficiently large $n$, we have
\[
    X \setminus W_{+} = x_{n}^{2}(X \setminus W_{+}) \subset x_{n}W_{-} \subset x_{n}(X \setminus W_{+}) \subset W_{-}.
\]
Hence, we obtain $x_{0} \in W_{-}$, which is a contradiction. 

Therefore, without loss of generality, we may assume that $w_{n}^{2}\neq 1$ for every $n\in \mathbb{N}$. Since
\[
g\{z_{+}, z_{-}\}\cap\{z_{+}, z_{-}\}\neq\emptyset
\] 
implies $g=1$, 
we can choose decreasing sequences of open neighborhoods
\[
(U_{n,+})_{n}, (U_{n,-})_{n}
\]
of $z_{+}$ and $z_{-}$, respectively, such that 
$U_{n,\pm}\subset V_{n,\pm}$
and 
 $w_{n}^{-2}U_{n,+}\cap U_{n,-}=\emptyset$.
 For each $n$, choose $k_{n}>n$ 
 so that 
 \[
 w_{k_{n}}(X\setminus U_{n,-})\subset U_{n,+},
 \]
 and define
 $z_{n}=w_{k_{n}}w_{n}^{-2}w_{k_{n}}$. 
 Since 
 \[
 w_{n}^{-2}U_{n,+}\cap U_{n, -}=\emptyset,
 \] 
 we have $z_{n}(X\setminus U_{n,-})
 \subset w_{k_{n}}w_{n}^{-2}U_{n,+}
 \subset w_{k_{n}}(X\setminus U_{n,-})
 \subset U_{n,+}$.
 Finally, let $W_{+}$ and $W_{-}$ be arbitrary open neighborhoods of $z_{+}$ and $z_{-}$ respectively.
 Since $U_{n, \pm}\subset W_{\pm}$ eventually, it follows that 
 \[
 z_{n}(X\setminus W_{-})
 \subset z_{n}(X\setminus U_{n, -})
 \subset U_{n,+}
 \subset W_{+}
 \]
  eventually. 
 \end{proof}

The following lemma generalizes \cite[Theorem 1]{Ozw25}. 
To prove this lemma, we adapt Ozawa's ``tree-graded'' space method.
\begin{lem}\label{thm:exZ}
Let $(A,\phi)$ be a separable $C^{*}$-probability space, and let $\Gamma$ be an infinite countable discrete group which has a topologically free, second countable extreme boundary. 
If $\sigma:\Gamma\curvearrowright A$ is an almost
periodic action, then the inclusion
\[
A\rtimes_{\sigma, r}\Gamma \subset A\rtimes_{\sigma*\id, r}(\Gamma * \mathbb{Z})
\]
is existential.

\end{lem}

\begin{proof}
By Lemma~\ref{thm:group}, there exists an axial sequence $(z_n)_n\subset\Gamma$ such that $\sigma_{z_n}\to\id$. Let $X$ be a topologically free extreme boundary associated with this sequence, with $|X|>2$.

Let $(\pi_\phi,\cH_\phi)$ be the GNS representation associated with $(A,\phi)$, and assume that $\pi_\phi$ is faithful.
The associated regular covariant representation $(\widetilde\pi_\phi,\lambda)$ on $\cH_\phi\otimes\ell^2(\Gamma)$ is defined by
\[
\widetilde\pi_\phi(a)(\zeta\otimes\delta_g)
=\pi_\phi(\sigma_{g^{-1}}(a))\zeta\otimes\delta_g,
\qquad
\lambda_h(\zeta\otimes\delta_g)=\zeta\otimes\delta_{hg}
\]
for $a\in A$, $\zeta\in\cH_\phi$, and $g,h\in\Gamma$.
Its integrated form yields a faithful representation
\[
\widetilde\pi_\phi\rtimes\lambda\colon
A\rtimes_{\sigma,r}\Gamma
\hookrightarrow
\mathcal B(\cH_\phi\otimes\ell^2(\Gamma)).
\]
Henceforth, we identify $A\rtimes_{\sigma,r}\Gamma$ with its image under this representation.

Let $\cU$ be a free ultrafilter on $\bN$. We define a $*$-representation
\[
\widetilde\pi\colon A\to(A\rtimes_{\sigma,r}\Gamma)^{\cU}\subset\mathcal B(\cH_\phi\otimes\ell^2(\Gamma))^{\cU}.
\]
We also define a unitary representation
\[
\widetilde\lambda\colon\Gamma*\langle z\rangle\to\mathcal U\bigl((A\rtimes_{\sigma,r}\Gamma)^{\cU}\bigr)
\]
by $\widetilde\lambda_g=[\lambda_g]_{\cU}$ for $g\in\Gamma$ and $\widetilde\lambda_z=[\lambda_{z_n}]_{\cU}$, where $\lambda_g(\zeta\otimes\delta_h)=\zeta\otimes\delta_{gh}$.

Represent $\mathcal B(\cH_\phi\otimes\ell^2(\Gamma))^{\cU}$ faithfully on a Hilbert space $\mathcal H$. Since $\sigma_{z_n}\to\id$ in the point-norm topology, we have
\[
\widetilde\lambda_z\widetilde\pi(a)\widetilde\lambda_z^*
=[\lambda_{z_n}\widetilde\pi_\phi(a)\lambda_{z_n}^*]_{\cU}
=[\widetilde\pi_\phi(\sigma_{z_n}(a))]_{\cU}
=\widetilde\pi(a).
\]
Together with the covariance relations for $\Gamma$, this shows that $(\widetilde\pi,\widetilde\lambda)$ is a covariant representation of $(A,\Gamma*\mathbb Z,\sigma*\id)$. Hence its integrated form defines a $*$-homomorphism
\[
\widetilde\pi\rtimes\widetilde\lambda\colon A\rtimes_{\sigma*\id}(\Gamma*\mathbb Z)\to(A\rtimes_{\sigma,r}\Gamma)^{\cU}\subset\mathcal B(\mathcal H).
\]

Let $T$ be the tree equipped with the $(\Gamma*\mathbb Z)$-action constructed in \cite{Ozw25}, and let $\overline T=T\cup\partial T$ be its compactification. Since $(1\otimes\ell^\infty(\Gamma))^{\cU}\subset\widetilde\pi(A)'$, where $\ell^\infty(\Gamma)$ acts on $\ell^2(\Gamma)$ by multiplication, the argument in the proof of \cite[Theorem~1]{Ozw25} yields a $(\Gamma*\mathbb Z)$-equivariant $*$-representation of $C(\overline T)$ on $\mathcal H$ whose range is contained in $\widetilde\pi(A)'$. 
The action on the vertices of $T$ is transitive, and the stabilizer of the vertex $\Gamma$ is 
$\Gamma$.
By \cite[Theorem 4]{Ozw25} and the continuity of
$(\widetilde\pi\rtimes\tilde{\lambda})|_{A\rtimes_{\sigma, r}\Gamma}$,
the representation $\widetilde\pi\rtimes\tilde{\lambda}$ is continuous on
$A\rtimes_{\sigma*\id, r}(\Gamma*\bZ)$.

For each $n$, let $\psi_n\colon\Gamma*\langle z\rangle\to\Gamma$ be the homomorphism that restricts to the identity on $\Gamma$ and sends $z$ to $z_n$. Since $|X|>2$, \cite[Lemma~6]{Ozw25} implies that, for every $w\in(\Gamma*\langle z\rangle)\setminus\{1\}$, we have $\psi_n(w)\neq1$ for all sufficiently large $n$. Identifying $A$ with its image under $\widetilde\pi_\phi$, we therefore obtain
\[
E_A^{\cU}\bigl((\widetilde\pi\rtimes\widetilde\lambda)(a\lambda_w)\bigr)
=[a\delta_{\psi_n(w),1}]_{\cU}
=\iota_A(a)\delta_{w,1}
\]
for all $a\in A$ and $w\in\Gamma*\langle z\rangle$, where $\iota_A\colon A\to A^{\cU}$ is the diagonal embedding. By linearity and continuity,
\[
E_A^{\cU}\circ(\widetilde\pi\rtimes\widetilde\lambda)=\iota_A\circ E_A.
\]
Since $E_A$ is faithful, $\widetilde\pi\rtimes\widetilde\lambda$ is injective. Its restriction to $A\rtimes_{\sigma,r}\Gamma$ is the diagonal embedding, so the inclusion is existential.

\end{proof}

If an infinite countable discrete group $\Gamma$ admits 
a topologically free extreme boundary, 
then by \cite[Corollary 2.2]{MO15}, $\Gamma*\mathbb{Z}$ is 
a non-elementary acylindrically hyperbolic group containing no nontrivial finite normal subgroup. 
Hence $\Gamma*\bZ$ admits an extreme boundary 
by Theorem \ref{thm:boundary}.
By applying Lemma~\ref{thm:exZ} twice 
and recalling that any infinite countable discrete group admitting a topologically free extreme boundary also admits a second countable one, we obtain the following:
\begin{cor}\label{thm:existential}
Let $(A,\phi)$ be a separable $C^{*}$-probability space, 
and let $\Gamma$ be an infinite countable discrete group which has a topologically free extreme boundary. 
If $\sigma:\Gamma\curvearrowright A$ is an almost
periodic action, then the inclusion
\[
A\rtimes_{\sigma,r}\Gamma \subset A\rtimes_{\sigma*\id,r}(\Gamma * \bF_{2})
\]
is existential.
\end{cor}

We denote by $(\cT,\omega)$ the Toeplitz $C^{*}$-probability space,
where $\cT$ is the Toeplitz algebra generated by the unilateral shift
$T$ on $\ell^{2}(\bN)$ and $\omega$ is the vacuum state associated
with $\delta_{0}$.

For a $C^{*}$-probability space $(A,\phi)$, let $\cT_{A,\phi}$
be the universal $C^{*}$-algebra generated by $A$ and an isometry $T$
subject to the relations
\[
    T^{*}aT=\phi(a) \qquad (a\in A).
\]
This is the Toeplitz--Pimsner algebra associated with the
$A$-correspondence $\overline{\operatorname{span}}(ATA)$.
There is a canonical state-preserving isomorphism
\[
    (\cT_{A,\phi},\phi\circ E_{A}^{\phi})
    \simeq (A,\phi)*(\cT,\omega),
\]
where $E_{A}^{\phi}\colon \cT_{A,\phi}\to A$ is the canonical
nondegenerate conditional expectation.
We refer the reader to \cite[Section~4.6]{BO08} for details.

Let $(\cC_{1},\omega)\subset(\cT,\omega)$ be the
$C^{*}$-probability space generated by
$s\coloneqq (T+T^{*})/2$, equipped with the restriction of $\omega$.
Let $\cO_{\infty}$ be the Cuntz algebra generated by isometries
$\{l_i:i\in\bN\}$ with mutually orthogonal ranges, and let $\omega$
denote its vacuum state.
Set
\[
    s_i\coloneqq (l_i+l_i^{*})/2 \qquad (i\in\bN).
\]
We denote by $(\cC,\omega)\subset(\cO_{\infty},\omega)$
the free semicircular $C^{*}$-probability space generated by
$\{s_i:i\in\bN\}$, equipped with the restriction of $\omega$.

For any $C^{*}$-algebra $A$, the conditional expectation
$\id_A\otimes\omega$ on $A\otimes_{\min}\cO_{\infty}$
is nondegenerate, since the simplicity of $\cO_{\infty}$
implies that the GNS representation associated with $\omega$ is faithful.
Moreover, its restriction to $A\otimes_{\min}\cC$ is faithful,
and hence nondegenerate, because $\omega|_{\cC}$ is faithful.

For each $2\leq d<\infty$, the reduced free group $C^{*}$-algebra
$C_r^{*}(\bF_d)$ and $\cC$ embed into one another;
see \cite[Section~4]{Oza}.

\begin{prop}\label{thm:verystrong}
Let $A$ be a unital, separable $C^*$-algebra, and let
$\alpha\colon\Gamma\curvearrowright A$ be an action of an infinite
countable discrete group.
Suppose that $A\rtimes_r\Gamma$ is simple and monotracial,
and let $\phi$ denote the unique $\Gamma$-invariant
tracial state on $A$.
Then there exist an ultrafilter $\mathcal U$ and an embedding
$\sigma$ making the following diagram commute:
\[
\begin{tikzcd}[column sep=large, row sep=large]
{(A\rtimes_r\Gamma,\phi\circ E_A)}
\arrow[r, hook, "\iota"]
\arrow[d, hook, "\theta"']
&
{\Bigl(
(A\rtimes_r\Gamma)*_A(A\otimes\cC),
\phi\circ\bigl(E_A*_A(\operatorname{id}\otimes\omega)\bigr)
\Bigr)^{\mathcal U}}
\\
{(A\rtimes_r\Gamma,\phi\circ E_A)*(\mathcal C_1,\omega)}
\arrow[ur, hook, "\sigma"']
&
\end{tikzcd}
\]
Here $\theta$ is the first-factor embedding, and $\iota$ is
the canonical embedding into the amalgamated free product
followed by the diagonal embedding into its ultrapower.
\end{prop}

\begin{proof}
It suffices to construct an embedding
\[
\sigma\colon
(A\rtimes_r\Gamma,\phi\circ E_A)*(\cT,\omega)
\hookrightarrow
\Bigl(
(A\rtimes_r\Gamma)*_A(A\otimes\cO_\infty),
\phi\circ\bigl(E_A*_A(\operatorname{id}\otimes\omega)\bigr)
\Bigr)^{\mathcal U}
\]
whose restriction to $A\rtimes_r\Gamma$ is the canonical inclusion into the amalgamated free product followed by the diagonal embedding into its ultrapower, and such that
\[
\sigma(T+T^*)\in\bigl((A\rtimes_r\Gamma)*_A(A\otimes\cC)\bigr)^{\mathcal U}.
\]

Indeed, \cite[Theorem~1.3]{BD01} gives the canonical inclusions
\[
(A\rtimes_r\Gamma,\phi\circ E_A)*(\cC_1,\omega)
\subset
(A\rtimes_r\Gamma,\phi\circ E_A)*(\cT,\omega)
\]
and
\[
\begin{aligned}
&\Bigl(
(A\rtimes_r\Gamma)*_A(A\otimes\cC),
\phi\circ\bigl(E_A*_A(\operatorname{id}\otimes\omega)\bigr)
\Bigr)
\\
&\qquad\subset
\Bigl(
(A\rtimes_r\Gamma)*_A(A\otimes\cO_{\infty}),
\phi\circ\bigl(E_A*_A(\operatorname{id}\otimes\omega)\bigr)
\Bigr).
\end{aligned}
\]
Since $\cC_1=C^*(T+T^*)\subset\cT$, the two conditions above imply that the restriction of $\sigma$ to $(A\rtimes_r\Gamma,\phi\circ E_A)*(\cC_1,\omega)$ takes values in $\bigl((A\rtimes_r\Gamma)*_A(A\otimes\cC)\bigr)^{\mathcal U}$ and thus gives the desired embedding.

We now construct such an embedding $\sigma$.
Since $A\rtimes_r\Gamma$ is simple and monotracial, 
the Dixmier property holds by the Haagerup-Zsido theorem \cite{HZ84}.
Namely, there exists a net
$(F_\lambda)_{\lambda\in\Lambda}$ of finite sequences of unitaries in
$A\rtimes_r\Gamma$ such that
\[
|F_\lambda|^{-1}\sum_{u\in F_\lambda}uau^*
\longrightarrow (\phi\circ E_A)(a)
\]
in norm for every $a\in A\rtimes_r\Gamma$.
Replacing $F_\lambda$ by $F_\lambda^*F_\lambda$, with multiplicities retained,
we may assume that $F_\lambda$ is invariant under taking adjoints.
Write $F_\lambda=(u_n^{(\lambda)})_{n=1}^{|F_\lambda|}$, and define
\[
G_\lambda=(u_1^{(\lambda)},iu_1^{(\lambda)},\ldots,
u_{|F_\lambda|}^{(\lambda)},iu_{|F_\lambda|}^{(\lambda)}),
\]
where repetitions are retained, so that $|G_\lambda|=2|F_\lambda|$.
Then
\[
|G_\lambda|^{-1}\sum_{u\in G_\lambda}uau^*
\longrightarrow (\phi\circ E_A)(a),
\qquad
|G_\lambda|^{-1}\sum_{u\in G_\lambda}uau=0,
\]
since $uau+(iu)a(iu)=0$.

Let $(v_n^{(\lambda)})_{n=1}^{|G_\lambda|}$ be an enumeration of $G_\lambda$.
For $1\leq n\leq |G_\lambda|$, define the self-adjoint element
\[
h_n^{(\lambda)}
\coloneqq (2|G_\lambda|)^{-1/2}
\bigl(v_n^{(\lambda)}+(v_n^{(\lambda)})^*\bigr)
\in A\rtimes_r\Gamma.
\]
Then
\[
\sum_n h_n^{(\lambda)}x h_n^{(\lambda)}
\longrightarrow (\phi\circ E_A)(x)
\]
for every $x\in A\rtimes_r\Gamma$.
Let $(l_n)_{n=1}^{\infty}$ be the canonical generators of $\cO_\infty$.
For each $\lambda$, consider the first $2|G_\lambda|$ generators,
which have mutually orthogonal ranges, and define
\[
T_\lambda\coloneqq \frac12\sum_n
\Bigl(
h_n^{(\lambda)}(1\otimes l_n)+(1\otimes l_n)h_n^{(\lambda)}
+i\bigl(h_n^{(\lambda)}(1\otimes l_{|G_{\lambda}|+n})
-(1\otimes l_{|G_{\lambda}|+n})h_n^{(\lambda)}\bigr)
\Bigr)
\]
in $(A\rtimes_r\Gamma)*_A(A\otimes\cO_\infty)$.

We next claim that
$(1\otimes l_n^*)a(1\otimes l_m)=E_A(a)\delta_{n,m}$
for every $a\in A\rtimes_r\Gamma$.
Since $E_A*_A(\operatorname{id}\otimes\omega)$ is nondegenerate,
it suffices to show that
\[
\bigl(E_A*_A(\operatorname{id}\otimes\omega)\bigr)
\bigl(x(1\otimes l_n^*)a(1\otimes l_m)y\bigr)=0
\]
for all $x,y\in(A\rtimes_r\Gamma)*_A(A\otimes\cO_\infty)$
and $a\in\ker E_A$.

Consider an element of the form
\[
x=b\lambda_g
\bigl(1\otimes l_{n_1}^{p_1}\cdots l_{n_s}^{p_s}
(l_{m_1}^{q_1})^*\cdots(l_{m_t}^{q_t})^*\bigr),
\]
where $b\in A$ and $g\in\Gamma$.
Here the indices in the monomial may range over all canonical generators
of $\cO_\infty$.
Then $x(1\otimes l_n^*)$ falls into one of the following four cases:
\begin{enumerate}
\item If $(p_1,\ldots,p_s,q_1,\ldots,q_t)\neq 0$ and $g\neq e$, then
$x(1\otimes l_n^*)\in(\ker E_A)(\ker(\operatorname{id}\otimes\omega))$.
\item If $(p_1,\ldots,p_s,q_1,\ldots,q_t)\neq 0$ and $g=e$, then
$x(1\otimes l_n^*)\in\ker(\operatorname{id}\otimes\omega)$.
\item If $(p_1,\ldots,p_s,q_1,\ldots,q_t)=0$ and $g\neq e$, then
$x(1\otimes l_n^*)\in(\ker E_A)(\ker(\operatorname{id}\otimes\omega))$.
\item If $(p_1,\ldots,p_s,q_1,\ldots,q_t)=0$ and $g=e$, then
$x(1\otimes l_n^*)\in\ker(\operatorname{id}\otimes\omega)$.
\end{enumerate}

We note that $(A\rtimes_r\Gamma)*_A(A\otimes\cO_\infty)$ is densely
spanned by products $z_1\cdots z_k$, each factor of which has the form
\[
b\lambda_g
\bigl(1\otimes l_{n_1}^{p_1}\cdots l_{n_s}^{p_s}
(l_{m_1}^{q_1})^*\cdots(l_{m_t}^{q_t})^*\bigr).
\]
For $k\geq 2$, the first factor $z_1$ has $g\in\Gamma$ and
$(p_1,\ldots,p_s,q_1,\ldots,q_t)\neq 0$;
the intermediate factors $z_2,\ldots,z_{k-1}$ have $g\neq e$ and
$(p_1,\ldots,p_s,q_1,\ldots,q_t)\neq 0$;
and the final factor $z_k$ has $g\neq e$, with the monomial in
$\cO_\infty$ allowed to be $1$.
Single-factor products are also included.
Thus, by linearity and continuity, it suffices to show that
\[
\bigl(E_A*_A(\operatorname{id}\otimes\omega)\bigr)
\bigl(x(1\otimes l_n^*)a(1\otimes l_m)y\bigr)=0
\]
when $x=x_1\cdots x_k$ and $y=y_1\cdots y_j$ are products of the above form.

By the preceding observations, $x(1\otimes l_n^*)$
is an alternating product of elements of $\ker E_A$ and
$\ker(\operatorname{id}\otimes\omega)$ ending with an element of
$\ker(\operatorname{id}\otimes\omega)$.
Similarly, $(1\otimes l_m)y$ is such an alternating product beginning
with an element of $\ker(\operatorname{id}\otimes\omega)$.
Since $a\in\ker E_A$, amalgamated freeness gives
\[
\bigl(E_A*_A(\operatorname{id}\otimes\omega)\bigr)
\bigl(x(1\otimes l_n^*)a(1\otimes l_m)y\bigr)=0.
\]
By linearity and continuity, this holds for all
$x,y\in(A\rtimes_r\Gamma)*_A(A\otimes\cO_\infty)$.
Consequently, $(1\otimes l_n^*)a(1\otimes l_m)=0$ for $a\in\ker E_A$.
For arbitrary $a\in A\rtimes_r\Gamma$, we therefore obtain
\[
\begin{aligned}
(1\otimes l_n^*)a(1\otimes l_m)
&=(1\otimes l_n^*)\bigl(a-E_A(a)+E_A(a)\bigr)(1\otimes l_m)\\
&=(1\otimes l_n^*)E_A(a)(1\otimes l_m)\\
&=E_A(a)(1\otimes l_n^*l_m)
=E_A(a)\delta_{n,m}.
\end{aligned}
\]

Hence, for $a\in A\rtimes_r\Gamma$, we have
\[
\begin{aligned}
&\Bigl(h_n^{(\lambda)}(1\otimes l_n)+(1\otimes l_n)h_n^{(\lambda)}
+i\bigl(h_n^{(\lambda)}(1\otimes l_{|G_\lambda|+n})
-(1\otimes l_{|G_\lambda|+n})h_n^{(\lambda)}\bigr)\Bigr)^*\\
&\quad {}\cdot a
\Bigl(h_m^{(\lambda)}(1\otimes l_m)+(1\otimes l_m)h_m^{(\lambda)}
+i\bigl(h_m^{(\lambda)}(1\otimes l_{|G_\lambda|+m})
-(1\otimes l_{|G_\lambda|+m})h_m^{(\lambda)}\bigr)\Bigr)\\
&=(1\otimes l_n^*)h_n^{(\lambda)}a h_m^{(\lambda)}(1\otimes l_m)
 +(1\otimes l_n^*)h_n^{(\lambda)}a(1\otimes l_m)h_m^{(\lambda)}\\
&\quad +h_n^{(\lambda)}(1\otimes l_n^*)a h_m^{(\lambda)}(1\otimes l_m)
 +h_n^{(\lambda)}(1\otimes l_n^*)a(1\otimes l_m)h_m^{(\lambda)}\\
&\quad +(1\otimes l_{|G_\lambda|+n}^*)h_n^{(\lambda)}a h_m^{(\lambda)}
 (1\otimes l_{|G_\lambda|+m})\\
&\quad -(1\otimes l_{|G_\lambda|+n}^*)h_n^{(\lambda)}a
 (1\otimes l_{|G_\lambda|+m})h_m^{(\lambda)}\\
&\quad -h_n^{(\lambda)}(1\otimes l_{|G_\lambda|+n}^*)a h_m^{(\lambda)}
 (1\otimes l_{|G_\lambda|+m})\\
&\quad +h_n^{(\lambda)}(1\otimes l_{|G_\lambda|+n}^*)a
 (1\otimes l_{|G_\lambda|+m})h_m^{(\lambda)}\\
&=2E_A\bigl(h_n^{(\lambda)}a h_m^{(\lambda)}\bigr)\delta_{n,m}
 +2h_n^{(\lambda)}E_A(a)h_m^{(\lambda)}\delta_{n,m}.
\end{aligned}
\]
Here the mixed terms involving $l_n$ and $l_{|G_\lambda|+m}$, or $l_{|G_\lambda|+n}$ and $l_m$, vanish by the preceding identities.
It follows that
\[
\begin{aligned}
T_\lambda^*aT_\lambda
&=\frac14\sum_n
\Bigl(2E_A\bigl(h_n^{(\lambda)}a h_n^{(\lambda)}\bigr)
+2h_n^{(\lambda)}E_A(a)h_n^{(\lambda)}\Bigr)\\
&=\frac12 E_A\Bigl(\sum_n h_n^{(\lambda)}a h_n^{(\lambda)}\Bigr)
+\frac12\sum_n h_n^{(\lambda)}E_A(a)h_n^{(\lambda)}.
\end{aligned}
\]
Since
$\sum_n h_n^{(\lambda)}a h_n^{(\lambda)}
\to(\phi\circ E_A)(a)$
and
$\sum_n h_n^{(\lambda)}E_A(a)h_n^{(\lambda)}
\to(\phi\circ E_A)(E_A(a))=(\phi\circ E_A)(a)$,
we conclude that
\[
T_\lambda^*aT_\lambda\longrightarrow(\phi\circ E_A)(a).
\]

Next, we show that
$\bigl(\phi\circ(E_A*_A(\operatorname{id}\otimes\omega))\bigr)
(aT_\lambda T_\lambda^*a^*)=0$
for every $a\in A\rtimes_r\Gamma$.
We claim that
\[
\bigl(E_A*_A(\operatorname{id}\otimes\omega)\bigr)
\bigl(x(1\otimes l_n)z(1\otimes l_m^*)y\bigr)=0
\]
for all $x,y,z\in A\rtimes_r\Gamma$.
Indeed,
\[
\begin{aligned}
&x(1\otimes l_n)z(1\otimes l_m^*)y\\
&=\bigl(x-E_A(x)+E_A(x)\bigr)(1\otimes l_n)
\bigl(z-E_A(z)+E_A(z)\bigr)(1\otimes l_m^*)
\bigl(y-E_A(y)+E_A(y)\bigr)\\
&=(x-E_A(x))(1\otimes l_n)(z-E_A(z))
(1\otimes l_m^*)(y-E_A(y))\\
&\quad +(x-E_A(x))(1\otimes l_n)(z-E_A(z))
(E_A(y)\otimes l_m^*)\\
&\quad +(x-E_A(x))(E_A(z)\otimes l_n)
(1\otimes l_m^*)(y-E_A(y))
 +(x-E_A(x))(E_A(z)E_A(y)\otimes l_nl_m^*)\\
&\quad +(E_A(x)\otimes l_n)(z-E_A(z))
(1\otimes l_m^*)(y-E_A(y))
 +(E_A(x)\otimes l_n)(z-E_A(z))(E_A(y)\otimes l_m^*)\\
&\quad +(E_A(x)E_A(z)\otimes l_nl_{m}^{*})
(y-E_A(y))
 +(E_A(x)E_A(z)E_A(y)\otimes l_nl_m^*).
\end{aligned}
\]
Each term has zero expectation by amalgamated freeness and the
defining properties of the vacuum state. Thus,
\[
\bigl(E_A*_A(\operatorname{id}\otimes\omega)\bigr)
\bigl(x(1\otimes l_n)z(1\otimes l_m^*)y\bigr)=0.
\]

Consequently, for every $a\in A\rtimes_r\Gamma$,
\[
\begin{aligned}
&\bigl(E_A*_A(\operatorname{id}\otimes\omega)\bigr)
\Bigl(a
\Bigl(h_n^{(\lambda)}(1\otimes l_n)+(1\otimes l_n)h_n^{(\lambda)}
+i\bigl(h_n^{(\lambda)}(1\otimes l_{|G_\lambda|+n})
-(1\otimes l_{|G_\lambda|+n})h_n^{(\lambda)}\bigr)\Bigr)\\
&\qquad {}\cdot
\Bigl(h_m^{(\lambda)}(1\otimes l_m)+(1\otimes l_m)h_m^{(\lambda)}
+i\bigl(h_m^{(\lambda)}(1\otimes l_{|G_\lambda|+m})
-(1\otimes l_{|G_\lambda|+m})h_m^{(\lambda)}\bigr)\Bigr)^*a^*\Bigr)=0.
\end{aligned}
\]
It follows that
$\bigl(\phi\circ(E_A*_A(\operatorname{id}\otimes\omega))\bigr)
(aT_\lambda T_\lambda^*a^*)=0$.

Finally, let $\cU$ be any cofinal ultrafilter on $\Lambda$. By \cite[Lemma~12]{Ozw25}, there exists an embedding $\sigma$ whose restriction to $A\rtimes_r\Gamma$ is the canonical inclusion followed by the diagonal embedding, and which satisfies $\sigma(T)=[T_\lambda]_{\cU}$.
Since
\[
T_\lambda+T_\lambda^*
\in (A\rtimes_r\Gamma)*_A(A\otimes\cC)
\]
for every $\lambda$, we have
\[
\sigma(T+T^*)
=[T_\lambda+T_\lambda^*]_{\cU}
\in
\bigl((A\rtimes_r\Gamma)*_A(A\otimes\cC)\bigr)^{\cU}.
\]
This completes the proof.

\end{proof}

\begin{rem}
The preceding statement and Lemma~\ref{thm:strong} are inspired by
the proof of \cite[Theorem~3]{Ozw25} and capture essentially the same phenomenon.
Indeed, when $\Gamma$ is trivial and $A$ is additionally assumed
to be exact and $\cZ$-stable, we obtain that $(A,\phi)$ is selfless.
The proof of the preceding statement was also inspired by
the choice of unitaries suggested by ChatGPT-6.0 Astra
for the construction presented in the appendix.
\end{rem}

Combining the state-preserving embedding
$(\cC,\omega)\hookrightarrow(C_r^*(\bF_2),\tau)$
with Proposition~\ref{thm:verystrong} and
\cite[Theorem~1.3]{BD01}, we obtain the following corollary.
\begin{cor}\label{thm:embed}
Let $(A,\phi)$ be a separable $C^*$-probability space, and let
$\alpha\colon\Gamma\curvearrowright A$ be an action of an infinite
countable discrete $C^*$-simple group.
Suppose that $A$ is $\Gamma$-simple and that $\phi$ is its
unique $\Gamma$-invariant tracial state.
Then there exist an ultrafilter $\cU$ and an embedding
$\sigma$ making the following diagram commute:
\[
\begin{tikzcd}[column sep=large, row sep=large]
{(A\rtimes_r\Gamma,\phi\circ E_A)}
\arrow[r, hook, "\iota"]
\arrow[d, hook, "\theta"']
&
{\Bigl(
(A\rtimes_r\Gamma)*_A(A\otimes C_r^*(\bF_2)),
\phi\circ\bigl(E_A*_A(\id\otimes\tau)\bigr)
\Bigr)^{\cU}}
\\
{(A\rtimes_r\Gamma,\phi\circ E_A)*(\cC_1,\omega)}
\arrow[ur, hook, "\sigma"']
&
\end{tikzcd}
\]
Here $\tau$ is the canonical trace on $C_r^*(\bF_2)$,
$\theta$ is the first-factor embedding, and $\iota$ is
the canonical embedding into the amalgamated free product
followed by the diagonal embedding into its ultrapower.
\end{cor}

\begin{thm}
Let $\Gamma$ be an infinite countable discrete group with a topologically free
extreme boundary, and let $\alpha\colon\Gamma\curvearrowright A$
be an almost periodic action on a unital, separable $C^*$-algebra.
Suppose that $A$ is $\Gamma$-simple and admits a unique
$\Gamma$-invariant tracial state $\phi$.
Then $(A\rtimes_{\alpha,r}\Gamma,\phi\circ E_A)$ is selfless.
\end{thm}

\begin{proof}

By Corollary~\ref{thm:embed}, we have
$
(A\rtimes_r\Gamma)*\mathcal C_1
\subset
((A\rtimes_{r}\Gamma)
*_A
(A\otimes C_r^*(\mathbb F_2)))^{\cU}.
$

On the other hand, by Corollary~\ref{thm:existential}, the inclusion
$A\rtimes_r\Gamma
\subset
(A\rtimes_{r}\Gamma)*_{A}(A\otimes C_r^*(\mathbb F_2))
$
is existential.
Therefore, the inclusion
$
A\rtimes_r\Gamma
\subset
(A\rtimes_r\Gamma)*\mathcal C_1
$
is existential.
Hence $\left(A\rtimes_{r}\Gamma, \phi\circ E_{A}\right)$ is selfless.

\end{proof}

\section{Examples}
In this section, we present concrete examples satisfying the hypotheses of the main theorem, with both commutative and noncommutative coefficient algebras.

\begin{examps}
\begin{enumerate}
\item
Choose an embedding $\iota\colon\bF_2\hookrightarrow\operatorname{SO}(3)$
with dense image.
Restricting the standard action of $\operatorname{SO}(3)$ on $S^2$
along $\iota$ induces an almost periodic action
\[
    \alpha\colon\bF_2\curvearrowright C(S^2).
\]
The underlying action on $S^2$ is minimal and uniquely ergodic.
Thus, integration against the normalized rotation-invariant measure
defines the unique $\alpha$-invariant tracial state $\tau$.
Consequently,
\[
    \bigl(C(S^2)\rtimes_{\alpha,r}\bF_2,\,
    \tau\circ E_{C(S^2)}\bigr)
\]
is selfless.

\item
Let $d,n\geq 1$, and write $\mathbb T=\mathbb R/\mathbb Z$
additively.
Choose $\theta_1,\ldots,\theta_d\in\mathbb R$ such that
$1,\theta_1,\ldots,\theta_d$ are linearly independent over $\mathbb Q$,
and set
\[
    X=\mathbb T^d\times\mathbb Z/2^n\mathbb Z.
\]
Writing $\bF_2=\langle a,b\rangle$, define an action on $X$ by
\[
\begin{aligned}
    a\cdot(x_1,\ldots,x_d,j)
        &=(x_1+\theta_1,\ldots,x_d+\theta_d,j),\\
    b\cdot(x_1,\ldots,x_d,j)
        &=(x_1,\ldots,x_d,j+1),
\end{aligned}
\]
where addition in the first $d$ coordinates is taken modulo $1$,
and addition in the last coordinate is taken modulo $2^n$.
This induces an almost periodic action
\[
    \alpha\colon\bF_2\curvearrowright C(X).
\]
The underlying action on $X$ is minimal,
and integration against the normalized Haar measure defines
the unique $\alpha$-invariant tracial state $\tau$.
Consequently,
\[
    \bigl(C(X)\rtimes_{\alpha,r}\bF_2,\,
    \tau\circ E_{C(X)}\bigr)
\]
is selfless.
Moreover, $X$ has covering dimension $d$, and the closure
of $\alpha(\bF_2)$ in the point-norm topology is isomorphic to
$\mathbb T^d\times\mathbb Z/2^n\mathbb Z$.
In particular, this closure has exactly $2^n$ connected components.

\item
Let $(A,\phi)$ be a unital, separable, simple, monotracial
$C^*$-probability space with $A\neq\bC$.
Let $\sigma\in\operatorname{Aut}(A\otimes A)$ be the flip automorphism,
defined by
\[
    \sigma(a\otimes b)=b\otimes a
    \qquad (a,b\in A).
\]
Choose a nontrivial homomorphism
$q\colon\bF_2\to\bZ/2\bZ$ and define
\[
    \gamma\colon\bF_2\curvearrowright A\otimes A,
    \qquad \gamma_g=\sigma^{q(g)}.
\]
Since $\gamma$ has finite image, it is almost periodic.
Consequently,
\[
    \bigl((A\otimes A)\rtimes_{\gamma,r}\bF_2,\,
    (\phi\otimes\phi)\circ E_{A\otimes A}\bigr)
\]
is selfless.
The image $\gamma(\bF_2)$ is finite, and $\gamma$ is not ergodic:
for any non-scalar self-adjoint element $a\in A$,
the element $a\otimes1+1\otimes a$ is a non-scalar fixed point.

\item
Let $\pi\colon\operatorname{SU}(2)\to\mathcal U(\bC^d)$
be a unitary representation, where $d\geq2$,
and let $\tau$ be the unique tracial state on $M_{d^\infty}$.
The infinite tensor product action
\[
    \rho=\bigotimes_{n=1}^{\infty}\operatorname{Ad}\pi
    \colon\operatorname{SU}(2)\curvearrowright M_{d^\infty}
\]
preserves $\tau$ and therefore induces the free product action
\[
    \rho*\rho\colon\operatorname{SU}(2)\curvearrowright
    (M_{d^\infty},\tau)*(M_{d^\infty},\tau).
\]
The underlying reduced free product $C^*$-algebra is unital,
separable, simple, and monotracial.
Choose an embedding $\iota\colon\bF_2\hookrightarrow\operatorname{SU}(2)$, and set
\[
    \delta_g=(\rho*\rho)_{\iota(g)}
    \qquad (g\in\bF_2).
\]
Then $\delta$ is almost periodic, and
\[
    \bigl((M_{d^\infty},\tau)*(M_{d^\infty},\tau)\bigr)
    \rtimes_{\delta,r}\bF_2
\]
is selfless.
\end{enumerate}
\end{examps}

\section{Appendix.}
In this appendix, we prove Lemma~\ref{thm:strong}, which establishes an embedding for reduced crossed products with simple, monotracial coefficient algebras. Although this lemma is not needed for the proof of the main theorem, it also yields Corollary~\ref{cor:direct-product} and may be of independent interest.

\subsection{Strong convergence of indexed families}

To prove Lemma \ref{thm:strong}, 
we rely on a result of Gao and Kunnawalkam Elayavalli \cite[Corollary 1.2]{GE26}. 
This result generalizes \cite[Corollary 4.3]{Pis16}, 
which in turn extends \cite[Theorem 3.1]{Sko15} 
(see the appendix of \cite{Mal12}). 
We note that our application of \cite[Corollary 1.2]{GE26} is inspired by the use of \cite[Corollary 4.3]{Pis16} in the proof of \cite[Theorem 1.9]{Rob25}.

Let us first recall the notion of strong convergence of indexed families.
We refer the reader to \cite{GE26} for a proper treatment.
Let $\Lambda$ be a directed set.
Suppose for each $\lambda\in\Lambda\cup\{\infty\}$, 
we have an inclusion of $C^{*}$-algebras 
$B^{(\lambda)}\subset A^{(\lambda)}$ admitting a nondegenerate conditional expectation $E^{(\lambda)}: A^{(\lambda)}\to B^{(\lambda)}$. 
Let $(a_{i}^{(\lambda)})_{i\in I}$ and $(a_{i}^{(\infty)})_{i\in I}$ 
be indexed families of elements in $C^{*}$-probability spaces
$(A^{(\lambda)}, E^{(\lambda)})$ and 
$(A^{(\infty)}, E^{(\infty)})$, respectively.
Assume further that $(a_{i}^{(\infty)})_{i\in I}$ generates $A^{(\infty)}$.
Let us say that $(a_{i}^{(\lambda)})_{i\in I}$ converges strongly to $(a_{i}^{(\infty)})_{i\in I}$, denoted
\[
\{ a_{i}^{(\lambda)} : i\in I\}\xlongrightarrow{s} \{ a_{i}^{\left(\infty\right)} : i\in I\},
\]
if for each cofinal ultrafilter $\cU$ on $\Lambda$, there exists an embedding of $C^{*}$-algebras
\[
\pi: A^{(\infty)}\to \prod_{\cU}A^{\left(\lambda\right)}
\]
such that
\begin{enumerate}
\item $\pi(a_{i}^{(\infty)})=[a_{i}^{(\lambda)}]_{\cU}$ for all $i\in I$; and
\item for any $a\in A^{(\infty)}$, one has $\pi(E^{(\infty)}(a))=[E^{(\lambda)}(a_{\lambda})]_{\cU}$, where $\pi(a)=[a_{\lambda}]_{\cU}$.
\end{enumerate}

We now clarify the precise definition of strong convergence for families of inclusions:

\begin{defn}
Fix an index set $J$ and an index set $I_j$ for each $j \in J$. 
Let $B^{\left(\lambda\right)} \subset A_{j}^{\left(\lambda\right)}$ be an inclusion of 
$C^*$-algebras with a nondegenerate conditional expectation $E_{j}^{\left(\lambda\right)}$; 
and let $(X_{i, j}^{\left(\lambda\right)})_{i \in I_j}$ be a family in $A_{j}^{\left(\lambda\right)}$, 
for each $j \in J$ and $\lambda \in \Lambda \cup \{\infty\}$. 
Assume further that $A_j\si$ is generated by $(X_{i, j}\si)_{i \in I_j}$ for every $j \in J$. 
We say that $(X_{i, j}^{\left(\lambda\right)}; E_{j}^{\left(\lambda\right)})_{i \in I_j} \to (X_{i, j}\si; E_j\si)_{i \in I_j}$ \emph{strongly for all $j \in J$, consistently on $B^{\left(\lambda\right)}$} 

if, for every cofinal ultrafilter $\cU$ on $\Lambda$, there exists an embedding
$\pi_j:A_j^{\left(\infty\right)}\to\prod_{\cU}A_j^{\left(\lambda\right)}$ for each $j\in J$ such that
\begin{enumerate}
    \item $\pi_j(X_{i,j}^{\left(\infty\right)})=[X_{i,j}^{\left(\lambda\right)}]_{\cU}$ for every $i\in I_j$ and $j\in J$;
    \item for $a\in A_j^{\left(\infty\right)}$ one has 
    $\pi_j(E_j^{\left(\infty\right)}(a))=[E_j^{(\lambda)}(a^{\left(\lambda\right)})]_{\cU}$,
    where $\pi_j(a)=[a^{\left(\lambda\right)}]_{\cU}$;
     \item $\pi_{j_1}(d)=\pi_{j_2}(d)$ for every $d\in B^{\left(\infty\right)}$ and every $j_1,j_2\in J$.
\end{enumerate}

\end{defn}

Although \cite{GE26} states the following result only for $\Lambda=\mathbb{N}$, the same conclusion holds for an arbitrary directed set $\Lambda$ by a similar argument.

\begin{thm}\label{thm:strong conve}\cite[Corollary 1.2]{GE26}
Under the same notation and assumptions as above, 
if $(X_{i, j}^{\left(\lambda\right)}; E_j^{\left(\lambda\right)})_{i \in I_j} \to (X_{i, j}\si; E_j\si)_{i \in I_j}$ 
strongly for all 
$j \in J$, consistently on 
$B^{\left(\lambda\right)}$, 
then $(X_{i, j}^{\left(\lambda\right)}; *_{B^{\left(\lambda\right)}}^{j \in J} E_j^{\left(\lambda\right)})_{i \in I_j, j \in J}$ in $*_{B^{\left(\lambda\right)}}^{j\in J} A_{j}^{\left(\lambda\right)}$ 
converges strongly to 
$(X_{i, j}\si; *_{B\si}^{j \in J} E_j\si)_{i \in I_j, j \in J}$ in $*_{B\si}^{j\in J} A_j\si$.
\end{thm}

\subsection{Proof of Lemma~\ref{thm:strong}}

In this subsection, we prove Lemma~\ref{thm:strong}.
We begin by establishing an isomorphism involving Toeplitz-Pimsner algebras.

\begin{lem}\label{thm:ToeplitzIsom}
Let $(A,\phi)$ be a $C^{*}$-probability space 
and let $\Gamma$ be an infinite countable discrete group. 
Then we have the following isomorphism:
\[
(\cT_{A\rtimes_{r}\Gamma, \phi\circ E_{A}}, E_{A}\circ E_{A\rtimes_{r}\Gamma}^{\phi\circ E_{A}})\simeq (\cT_{A,\phi}, E_{A}^{\phi})*_{A}(A\rtimes_{r}\Gamma, E_{A}).
\]

\end{lem}

\begin{proof}

We claim that $T^{*}a\lambda_{g}T = \left(\phi\circ E_{A}\right) (a\lambda_{g})$. Since $E_{A}^{\phi}*_{A}E_{A}$ is nondegenerate, it suffices to show that
\[
  (E_{A}^{\phi}*_{A}E_{A})(xT^{*}a\lambda_{g}T y) = 0
\]
for all $x, y \in \cT_{A, \phi}*_A (A\rtimes_{r}\Gamma)$ and $g \neq e$. 
    
Let 
\[
  z = T^{p_{0}}a_1T^{p_1}a_{2}\cdots a_{k}T^{p_{k}}(T^{*})^{q_{0}}b_{1}(T^{*})^{q_{1}}b_{2}(T^{*})^{q_2}\cdots b_{l}(T^{*})^{q_l}c\lambda_{g},
\]
where $g \in \Gamma $. 
Then $z$  falls into one of the following four cases:
\begin{enumerate}
  \item If $(p_0,\cdots p_k, q_{0},\cdots q_{l})= 0$ and $g \neq e$, then $z \in \Ker E_{A}$.
  \item If $(p_0,\cdots p_k, q_{0},\cdots q_{l})= 0$ and $g = e$, then $z \in A$.
  \item If $(p_0,\cdots p_k, q_{0},\cdots q_{l})\neq 0$ and $g \neq e$, then $z \in (\Ker E_{A}^{\phi})(\Ker E_{A})$.
  \item If $(p_0,\cdots p_k, q_{0},\cdots q_{l})\neq 0$ and $g = e$, then $z \in \Ker E_{A}^{\phi}$.
\end{enumerate}

Consequently, we obtain:
\begin{enumerate}
\item If $(p_0,\cdots p_k, q_{0},\cdots q_{l})= 0$ and $g \neq e$, then $zT^{*}\in (\Ker E_{A})(\Ker E_{A}^{\phi}), Tz\in (\Ker E_{A}^{\phi})(\Ker E_{A})$.
\item If $(p_0,\cdots p_k, q_{0},\cdots q_{l})= 0$ and $g = e$, then $zT^{*}\in \Ker E_{A}^{\phi}, Tz\in \Ker E_{A}^{\phi}$.
\item If $(p_0,\cdots p_k, q_{0},\cdots q_{l})\neq 0$ and $g \neq e$, then $zT^{*}\in(\Ker E_{A}^{\phi})(\Ker E_{A})(\Ker E_{A}^{\phi}), Tz\in (\Ker E_{A}^{\phi})(\Ker E_{A})$.
\item If $(p_0,\cdots p_k, q_{0},\cdots q_{l})\neq 0$ and $g = e$, then $zT^{*}\in\Ker E_{A}^{\phi}, Tz\in\Ker E_{A}^{\phi}$
\end{enumerate}

We note that $\cT_{A,\phi}*_A(A\rtimes_r\Gamma)$
is densely spanned by products $z_1\cdots z_n$,
where each factor has the form
\[
T^{p_0}a_1T^{p_1}a_2\cdots a_kT^{p_k}
(T^*)^{q_0}b_1(T^*)^{q_1}b_2\cdots b_l(T^*)^{q_l}
c\lambda_g.
\]
For $n\geq2$, we may impose the following conditions:
the first factor $z_1$ has $g\neq e$, with
$(p_0,\ldots,p_k,q_0,\ldots,q_l)$ allowed to be zero;
the intermediate factors $z_2,\ldots,z_{n-1}$ have
$g\neq e$ and $(p_0,\ldots,p_k,q_0,\ldots,q_l)\neq0$;
and the final factor $z_n$ has
$(p_0,\ldots,p_k,q_0,\ldots,q_l)\neq0$, with $g\in\Gamma$
arbitrary.
Single-factor products are also included, with no restrictions
on $g$ or on $(p_0,\ldots,p_k,q_0,\ldots,q_l)$.

Thus, by linearity and continuity, it suffices to show that
\[
(E_A^\phi*_A E_A)(xT^*a\lambda_gTy)=0
\]
when $x=x_1\cdots x_n$ and $y=y_1\cdots y_m$
are products of the above form.

   By the above observations, $xT^{*}$ (resp.\ $Ty$) is an alternating product of elements in $\operatorname{ker} E_{A}$ and $\operatorname{ker} E_{A}^{\phi}$, ending (resp.\ beginning) with an element of $\operatorname{ker} E_{A}^{\phi}$.

Hence, we obtain $(E_{A}^{\phi} \ast_A E_{A})(x T^* a \lambda_g T y) = 0$ for all $x, y \in \cT_{A, \phi} \ast_A (A \rtimes_r \Gamma)$. Consequently, it follows that $T^* a \lambda_g T = 0$ for all $g \neq e$.
  
  The universal $C^*$-algebra generated by $A \rtimes_r \Gamma$ and an isometry $T$ satisfying $T^* x T = \phi \circ E_A(x)$ for all $x \in A \rtimes_r \Gamma$ is the Toeplitz--Pimsner algebra $\cT_{A \rtimes_r \Gamma, \phi \circ E_A}$. Hence, we obtain a canonical surjection $\Phi: \cT_{A \rtimes_r \Gamma, \phi \circ E_A} \to \cT_{A, \phi} \ast_A (A \rtimes_r \Gamma)$.

We will show that $\Phi$ is an embedding. First, we observe that
\[
  (E_{A}^{\phi} \ast_A E_{A}) \circ \Phi = E_{A} \circ E_{A \rtimes_r \Gamma}^{\phi \circ E_A}.
\]

Indeed, let
\[
z=T^{p_0}a_1\lambda_{g_1}T^{p_1}\cdots a_k\lambda_{g_k}T^{p_k}(T^*)^{q_0}a_{k+1}\lambda_{g_{k+1}}(T^*)^{q_1}\cdots a_{k+l}\lambda_{g_{k+l}}(T^*)^{q_l}
\in\cT_{A,\phi}*_A(A\rtimes_r\Gamma),
\]
where $a_1,\ldots,a_{k+l}\in A$, $p_i\neq0$ for $1\leq i\leq k-1$, $q_j\neq0$ for $1\leq j\leq l-1$, and $(p_k,q_0)\neq0$. We distinguish the following four cases:
\begin{enumerate}
\item If $g_i=e$ for all $i$, then $z=T^{p_0}a_1T^{p_1}\cdots a_kT^{p_k}(T^*)^{q_0}a_{k+1}(T^*)^{q_1}\cdots a_{k+l}(T^*)^{q_l}\in\ker E_A^\phi$.
\item If $i_0>1$ is the first index such that $g_{i_0}\neq e$, then $z=x_1a_{i_0}\lambda_{g_{i_0}}y_1$, where $x_1\in\ker E_A^\phi$ and $y_1$ is the remaining suffix of the displayed expression for $z$.
\item If $g_1\neq e$ and $p_0\neq0$, then $z=x_1a_1\lambda_{g_1}y_1$, where $x_1=T^{p_0}\in\ker E_A^\phi$ and $y_1$ is the remaining suffix, beginning with a nonzero power of $T$ or $T^*$.
\item If $g_1\neq e$ and $p_0=0$, then $z=a_1\lambda_{g_1}y_1$, where $y_1$ is the remaining suffix, beginning with a nonzero power of $T$ or $T^*$.
\end{enumerate}

By induction, $z$ is an alternating product of elements of $\ker E_A^\phi$ and $\ker E_A$, and hence has zero expectation. Writing $b_j=a_{k+j}$ and $h_j=g_{k+j}$ for $1\leq j\leq l$, and including the immediate case in which all exponents vanish, we obtain
\[
\begin{aligned}
&(E_A^\phi*_A E_A)
\Bigl(T^{p_0}a_1\lambda_{g_1}T^{p_1}\cdots a_k\lambda_{g_k}T^{p_k}
(T^*)^{q_0}b_1\lambda_{h_1}(T^*)^{q_1}\cdots b_l\lambda_{h_l}(T^*)^{q_l}\Bigr)\\
&\quad=
\begin{cases}
E_A(a_1\lambda_{g_1}\cdots a_k\lambda_{g_k}b_1\lambda_{h_1}\cdots b_l\lambda_{h_l}),
&\text{if }p_i=0\text{ and }q_j=0\text{ for all }i,j,\\[1.5ex]
0,&\text{otherwise}.
\end{cases}
\end{aligned}
\]
On the other hand, we have
\[
\begin{aligned}
  &(E_{A} \circ E_{A \rtimes_r \Gamma}^{\phi \circ E_A})
  \Bigl( T^{p_{0}} a_{1} \lambda_{g_{1}} T^{p_{1}} \dots a_{k} \lambda_{g_{k}} T^{p_{k}}
  (T^{*})^{q_{0}} b_{1} \lambda_{h_{1}} (T^{*})^{q_{1}} \dots b_{l} \lambda_{h_{l}} (T^{*})^{q_{l}} \Bigr) \\
  &\quad =
  \begin{cases}
    E_{A}( a_{1} \lambda_{g_{1}} \dots a_{k} \lambda_{g_{k}} b_{1} \lambda_{h_{1}} \dots b_{l} \lambda_{h_{l}} ),
    & \text{if } p_{i} = 0 \text{ and } q_{j} = 0 \text{ for all } i, j, \\[1.5ex]
    0,
    & \text{otherwise}.
  \end{cases}
\end{aligned}
\] 
  Let $a \in \ker \Phi$. Since $(E_{A}^{\phi} \ast_A E_{A}) \circ \Phi = E_{A} \circ E_{A \rtimes_r \Gamma}^{\phi \circ E_A}$, we have
\[
  (E_{A} \circ E_{A \rtimes_r \Gamma}^{\phi \circ E_A})(x a y) = 0
\]
for all $x, y \in \cT_{A \rtimes_r \Gamma, \phi \circ E_A}$. In particular, by Kadison's inequality,
\[
0 \leqslant E_{A} \Bigl( E_{A \rtimes_r \Gamma}^{\phi \circ E_A}(x a y)^{*} E_{A \rtimes_r \Gamma}^{\phi \circ E_A}(x a y) \Bigr) \leqslant (E_{A} \circ E_{A \rtimes_r \Gamma}^{\phi \circ E_A}) \Bigl( (x a y)^{*}(x a y) \Bigr) = 0
\]
for all $x, y \in \cT_{A \rtimes_r \Gamma, \phi \circ E_A}$. 
Since $E_{A}$ is faithful, it follows that $E_{A \rtimes_r \Gamma}^{\phi \circ E_A}(x a y) = 0$. By the nondegeneracy of $E_{A \rtimes_r \Gamma}^{\phi \circ E_A}$, we obtain $a = 0$. Thus, $\Phi$ is an embedding.

\end{proof}

We now use the isomorphism established in Lemma~\ref{thm:ToeplitzIsom} to prove
the main result of this subsection.

\begin{lem}\label{thm:strong}
Let $(A,\phi)$ be a separable, simple, monotracial
$C^*$-probability space, and let
$\alpha\colon\Gamma\curvearrowright A$ be an action of an infinite
countable discrete group.
Then there exist an ultrafilter $\mathcal U$ and an embedding
$\sigma$ making the following diagram commute:
\[
\begin{tikzcd}[column sep=large, row sep=large]
{(A\rtimes_r\Gamma,\phi\circ E_A)}
\arrow[r, hook, "\iota"]
\arrow[d, hook, "\theta"']
&
{\Bigl(
(A\rtimes_r\Gamma)*_A(A\otimes\cC),
\phi\circ\bigl(E_A*_A(\operatorname{id}\otimes\omega)\bigr)
\Bigr)^{\mathcal U}}
\\
{(A\rtimes_r\Gamma,\phi\circ E_A)*(\mathcal C_1,\omega)}
\arrow[ur, hook, "\sigma"']
&
\end{tikzcd}
\]
Here $\theta$ is the first-factor embedding, and $\iota$ is
the canonical embedding into the amalgamated free product
followed by the diagonal embedding into its ultrapower.
\end{lem}

\begin{proof}
It suffices to construct an embedding
\[
\Psi\colon
\left((A\rtimes_r\Gamma)*_A\cT_{A,\phi},\,E_A*_A E_A^\phi\right)
\hookrightarrow
\Bigl((A\rtimes_r\Gamma)*_A(A\otimes\cO_\infty),\,E_A*_A(\operatorname{id}\otimes\omega)\Bigr)^{\cU}
\]
whose restriction to $A\rtimes_r\Gamma$ is the canonical inclusion followed by the diagonal embedding, and such that
\[
\Psi(T+T^*)\in\bigl((A\rtimes_r\Gamma)*_A(A\otimes\cC)\bigr)^{\cU}.
\]
Indeed, by Lemma~\ref{thm:ToeplitzIsom}, the domain is canonically identified with $\cT_{A\rtimes_r\Gamma,\phi\circ E_A}$. Since $\cC_1=C^*(T+T^*)\subset\cT$, restricting $\Psi$ to $(A\rtimes_r\Gamma,\phi\circ E_A)*(\cC_1,\omega)$ then yields the desired embedding. Here we use the canonical inclusion
\[
(A\rtimes_r\Gamma)*_A(A\otimes\cC)\subset(A\rtimes_r\Gamma)*_A(A\otimes\cO_\infty),
\]
which follows from \cite[Theorem~1.3]{BD01} and the nondegeneracy of the relevant conditional expectations.

To construct such an embedding, we apply Theorem~\ref{thm:strong conve}.
To this end, we first construct indexed families of elements in $A\rtimes_{r}\Gamma$ (resp. $A \otimes \cO_{\infty}$) that converge strongly to an indexed family of generators of $A\rtimes_{r}\Gamma$ (resp. $\cT_{A,\phi}$) consistently on $A$.

To begin with, we choose a countable generating set $\{a_i : i \in I\}$ of $A$. Then the set
\[
X_1^{(\lambda)} = X_1^{(\infty)} \coloneqq \{a_i : i \in I\} \cup \{\lambda_s : s \in \Gamma\}
\]
generates $A \rtimes_{r} \Gamma$.

Next, since $A$ is a simple, monotracial $C^{*}$-algebra, 
there is a net $(F_{\lambda})_{\lambda \in \Lambda}$ of finite sequences of unitary elements in $A$ such that the elements
$|F_{\lambda}|^{-1}\sum_{u\in F_{\lambda}}uau^{*}\to\phi(a)$ for all $a\in A$.
By replacing $F_{\lambda}$ with $F_{\lambda}^{*}F_{\lambda}$, we may assume that the finite sequence $F_{\lambda}$ is closed under the adjoint operation.
Write $F_\lambda=(u_n^{(\lambda)})_{n=1}^{|F_\lambda|}$, and let $G_\lambda$ be the finite sequence of unitaries in $A$ given by
\[
G_\lambda=(u_1^{(\lambda)},iu_1^{(\lambda)},\ldots,
u_{|F_\lambda|}^{(\lambda)},iu_{|F_\lambda|}^{(\lambda)}).
\]
Repetitions are retained, so that $|G_\lambda|=2|F_\lambda|$.
Then we have $|G_{\lambda}|^{-1}\sum_{u\in G_{\lambda}}uau^{*}\to\phi(a)$
and $|G_{\lambda}|^{-1}\sum_{u\in G_{\lambda}}uau=0$,
 because $uau+(iu)a(iu)=0$.
Write $G_\lambda=(v_n^{(\lambda)})_{n=1}^{|G_\lambda|}$, and let $(l_n)_{n=1}^{\infty}$ be the canonical generators of $\cO_\infty$, which are isometries with mutually orthogonal ranges.
Then the elements
\[
T_\lambda \coloneqq (2|G_\lambda|)^{-1/2}
\sum_{n}
\bigl(v_n^{(\lambda)}+(v_n^{(\lambda)})^*\bigr)\otimes l_n
\]
in $A \otimes \mathcal O_\infty$ satisfy $T_\lambda^* a T_\lambda \to \phi(a)$ and $(\phi \otimes \omega)(a T_{\lambda} T_{\lambda}^* a^*) = 0$ for all $a \in A$.
In fact, 
\[
\begin{aligned}
T_{\lambda}^{*}(a\otimes 1)T_{\lambda}
&=(2|G_{\lambda}|)^{-1}\sum_{n}((v_{n}^{(\lambda)}+(v_{n}^{(\lambda)})^{*})a(v_{n}^{(\lambda)}+(v_{n}^{(\lambda)})^{*}))\otimes 1\\
&=|G_{\lambda}|^{-1}\sum_{n}(v_{n}^{(\lambda)}a(v_{n}^{(\lambda)})^{*} )\otimes 1.
\end{aligned}
\]
Moreover, we have $(T_{\lambda} + T_{\lambda}^*)/2 \in A \otimes \cC$. Hence, for any cofinal ultrafilter $\cU$ on $\Lambda$, we obtain an embedding
\[
\pi \colon (A, \phi) \ast (\cT, \omega) \to (A \otimes \cO_\infty, \phi \otimes \omega)^\cU
\]
such that $\pi(T) = [T_{\lambda}]_\cU$.

We define
\[
X_2^{(\lambda)} \coloneqq \{a_i : i \in I\} \cup \{T_{\lambda}\} \quad \text{and} \quad X_2^{(\infty)} \coloneqq \{a_i : i \in I\} \cup \{T\}.
\]
Then $X_{2}^{(\infty)}$ generates $\cT_{A,\phi}$.

Subsequently,
we show that $(X_{i}^{(\lambda)})_{\lambda}\xlongrightarrow{s}(X_{i}^{(\infty)})_{\lambda}$ consistently on $A$ for $i=1,2$.

We remark that
\[
  \cT_{A, \phi} = \overline{\operatorname{span}} \bigl\{ T^{p_{1}} a_1 T^{p_2} a_{2} \dots a_{k} T^{p_{k}} (T^{*})^{q_{0}} b_{1} (T^{*})^{q_{1}} b_{2} (T^{*})^{q_2} \dots b_{l} (T^{*})^{q_l} : a_{i}, b_{j} \in A,\, p_{i}, q_{j} \in \bN \bigr\}.
\]

It holds that
\[
\begin{aligned}
&E_{A}^{\phi}
\Bigl(
T^{p_{0}} a_1 T^{p_1} a_{2} \dots a_{k} T^{p_{k}} (T^{*})^{q_{0}} b_{1} (T^{*})^{q_{1}} b_{2} (T^{*})^{q_2} \dots b_{l} (T^{*})^{q_l}
\Bigr) \\
&=
\begin{cases}
a_{1} \dots a_{k} b_{1} \dots b_{l},
& \text{if } p_{i} = 0 \text{ and } q_{j} = 0 \text{ for all } i, j, \\[1.5ex]
0,
& \text{otherwise}.
\end{cases}
\end{aligned}
\]
Since $\omega$ is the vacuum state on $\cO_{\infty}$, we also have
\[
\begin{aligned}
&(\operatorname{id} \otimes \omega)
\Bigl(
T_{\lambda}^{p_{0}} (a_{1} \otimes 1) T_{\lambda}^{p_1} (a_{2} \otimes 1) \dots (a_{k} \otimes 1) T_{\lambda}^{p_{k}} (T_{\lambda}^{*})^{q_{0}} (b_{1} \otimes 1) (T_{\lambda}^{*})^{q_{1}} (b_{2} \otimes 1) (T_{\lambda}^{*})^{q_2} \dots (b_{l} \otimes 1) (T_{\lambda}^{*})^{q_l}
\Bigr) \\
&=
\begin{cases}
a_{1} \dots a_{k} b_{1} \dots b_{l},
& \text{if } p_{i} = 0 \text{ and } q_{j} = 0 \text{ for all } i, j, \\[1.5ex]
0,
& \text{otherwise}.
\end{cases}
\end{aligned}
\]

Hence, for $x \in \cT_{A, \phi}$, if $\pi(x) = [x_{\lambda}]_{\cU}$, then
\[
  \pi(E_{A}^{\phi}(x)) = [(\operatorname{id} \otimes \omega)(x_{\lambda})]_{\cU}.
\]
Therefore, we have $(X_{1}^{(\lambda)}, E_{A}) \to (X_{1}^{(\infty)}, E_{A})$ and
$(X_{2}^{(\lambda)}, \operatorname{id} \otimes \omega) \to (X_{2}^{(\infty)}, E_{A}^{\phi})$ strongly consistently on $A$.

By Theorem~\ref{thm:strong conve}, it holds that $(X_{1}^{(\lambda)}, X_{2}^{(\lambda)}, E_{A}*_{A}(\operatorname{id} \otimes \omega))$ in 
$(A\rtimes_{r}\Gamma)*_{A}(A \otimes \cO_{\infty})$ 
converges strongly to $(X_{1}^{(\infty)}, X_{2}^{(\infty)}, E_A \ast_A E_A^\phi)$ in $(A \rtimes_r \Gamma) \ast_A \cT_{A, \phi}$. 
Hence there exists an embedding
\[
\left((A \rtimes_r \Gamma) \ast_A \cT_{A,\phi}, E_{A}*_{A}E_{A}^{\phi}\right) \hookrightarrow \Bigl( (A \rtimes_r \Gamma) \ast_A (A \otimes \cO_{\infty}), \, E_A \ast_A (\operatorname{id} \otimes \omega) \Bigr)^\cU.
\]

Denote this embedding by $\Psi$. By construction, its restriction to $A\rtimes_r\Gamma$ is the canonical inclusion followed by the diagonal embedding, and $\Psi(T)=[T_\lambda]_{\cU}$. Since $T_\lambda+T_\lambda^*\in A\otimes\cC$ for every $\lambda$, $\Psi$ satisfies the required conditions. This completes the proof. 

 \end{proof}
  
  As an immediate consequence of the preceding lemma, we recover
the following special case of \cite[Corollary~7.9]{GJEPR26},
which also follows from Proposition~\ref{thm:verystrong}.
 
 \begin{cor}\label{cor:direct-product} 
Let $G$ and $H$ be infinite countable discrete groups.
Suppose that $G$ is exact and $C^{*}$-simple, and that $H$ is $C^{*}$-selfless.
Then $G\times H$ is $C^{*}$-selfless.
\end{cor}

 \begin{proof}
Since $G$ is $C^*$-simple, the preceding lemma (or Proposition~\ref{thm:verystrong}) yields
a state-preserving embedding

\[
\begin{aligned}
  &\bigl(C_{r}^{*}(G)\otimes C_{r}^{*}(H), \, \tau_{G}\otimes \tau_{H}\bigr) * (\mathcal{C}_1, \omega) \\
  &\subset \bigl(\left(C_{r}^{*}(G)\otimes C_{r}^{*}(H)\right) *_{C_{r}^{*}(G)} (C_{r}^{*}(G) \otimes\cC), \, \tau_{G} \circ ((\operatorname{id}\otimes\tau_{H})*_{C_{r}^{*}(G)}(\operatorname{id} \otimes \omega) )\bigr)^{\mathcal{U}}\\
&\simeq(C_{r}^{*}(G)\otimes (C_{r}^{*}(H)*\cC), \tau_{G}\otimes(\tau_{H}*\omega))^{\cU}.
 \end{aligned}
\] 
whose restriction to $C_r^*(G)\otimes C_r^*(H)$
is the diagonal embedding.

Since $H$ is $C^*$-selfless, the canonical inclusion
\[
C_r^*(H)\subset C_r^*(H)*\mathcal C
\]
is existential.
Moreover, exactness of $C_r^*(G)$ gives a canonical embedding
\[
C_r^*(G)\otimes\bigl(C_r^*(H)\bigr)^{\mathcal V}
\hookrightarrow
\bigl(C_r^*(G)\otimes C_r^*(H)\bigr)^{\mathcal V}.
\]
It follows that the canonical inclusion
\[
C_r^*(G)\otimes C_r^*(H)
\subset C_r^*(G)\otimes(C_r^*(H)*\mathcal C)
\]
is existential.
Combining this with the preceding embedding, and identifying
iterated ultrapowers with an ultrapower, we conclude that
the canonical inclusion
\[
C_r^*(G)\otimes C_r^*(H)
\subset \bigl(C_r^*(G)\otimes C_r^*(H)\bigr)*\mathcal C_1
\]
is existential.
This proves the assertion.
\end{proof}

\vspace{\baselineskip}

\end{document}